\documentclass[11pt, reqno]{amsart}%

\usepackage{amsmath}
\usepackage{array}
\usepackage{amssymb}

\usepackage{graphicx}
\usepackage{subfigure}                                                                                          
\usepackage[bookmarks=false]{hyperref}
\usepackage{epstopdf}

\def\+{{\oplus}}

\newtheorem{thm}{Theorem}[section]
\newtheorem{theorem}{Theorem}[section]
\newtheorem{lemma}[thm]{Lemma}
\newtheorem{cor}[thm]{Corollary}
\newtheorem{conj}[thm]{Conjecture}
\newtheorem{prop}[thm]{Proposition}

\newtheorem{defn}[thm]{Definition}
\newtheorem{example}[thm]{Example}
\newtheorem{remark}[thm]{Remark}

\numberwithin{equation}{section}

\begin{document}

%
\title[Boolean nested canalizing functions]{Boolean nested canalizing functions: a comprehensive analysis}

\author[Y. Li, J. O. Adeyeye, D. Murrugarra, B. Aguilar, R. Laubenbacher]{Yuan Li$^{1\ast}$, John O. Adeyeye $^{2\ast}$, David Murrugarra$^{3\ast\ast}$, Boris Aguilar$^{4}$, Reinhard Laubenbacher$^{5\ast\ast}$}
\address{{\small $^{1}$Department of Mathematics, Winston-Salem State University, NC
27110,USA}\\
{\small email: liyu@wssu.edu }\\
$^{2}$Department of Mathematics, Winston-Salem State University, NC 27110,USA,
{\small email: adeyeyej@wssu.edu}\\
{\small $^{3}$ School of Mathematics, Georgia Institute of Technology, Atlanta, GA 30332-0160 USA}\\
{\small email: davidmur@math.gatech.edu}\\
{\small $^{4}$Department of Computer Science, Virginia Tech
Blacksburg, VA 24061-0123, USA}\\
{\small email: baguilar@vt.edu}\\
{\small $^{5}$ Virginia Bioinformatics Institute, Virginia Tech, Blacksburg, VA
24061,USA and Department of Mathematics, Virginia Tech, Blacksburg, VA 24061-0123, USA}\\
{\small email: reinhard@vbi.vt.edu}}
\thanks{$^{\ast}$ Supported by Award $\#$ W911NF-11-10166 from the US DoD. $^{\ast\ast}$ 
Supported by NSF Grant CMMI-0908201.}
\date{}

\begin{abstract}
Boolean network models of molecular regulatory networks have been used successfully in computational systems biology.
The Boolean functions that appear in published models tend to have special properties, in particular the property
of being nested canalizing, a concept inspired by the concept of canalization in evolutionary biology. 
It has been shown that networks comprised of nested canalizing functions have dynamic properties that
make them suitable for modeling molecular regulatory networks, namely a small number of (large) attractors,
as well as relatively short limit cycles.

This paper contains a detailed analysis of this class of functions, based on a novel normal form as
polynomial functions over the Boolean field. The concept of layer is introduced that stratifies variables into
different classes depending on their level of dominance. Using this layer concept a closed form formula
is derived for the number of nested canalizing functions with a given number of variables. Additional
metrics considered include 
Hamming weight, the activity number of any variable, and the average sensitivity of the function. 
It is also shown that the average sensitivity of any nested canalizing function is between 0 and 2. 
This provides a rationale for why nested canalizing functions are stable, 
since a random Boolean function in $n$ variables has average sensitivity $\frac{n}{2}$. 
The paper also contains experimental evidence that the layer number is an important factor in network
stability. 
\end{abstract}
\maketitle

\interdisplaylinepenalty=2500

\section{Introduction}

\label{sec-intro} Canalizing Boolean functions were introduced by S. Kauffman and collaborators \cite{Kau1}
as appropriate rules in Boolean network models of gene regulatory networks.
More recently, a subclass of these functions, so-called 
\emph{nested canalizing} functions (NCF) was introduced \cite{Kau2} 
and studied from the point of view of stability properties of network dynamics. 
A multi-state version of such functions has been introduced in \cite{Mur,Mur2}, where it
was shown that networks whose dynamics are controlled by nested canalizing functions
have similar stability properties, namely large attractor basins and short limit cycles. 
An analysis of published models, both Boolean and multi-state, of molecular regulatory
networks revealed that the large majority of regulatory rules in them is canalizing, with most of these in fact
nested canalizing \cite{Har, Kau3, Nik, Mur}. Thus, nested canalizing rules and
the properties of networks governed by them are important to study because of their relevance in systems biology. Furthermore,
they are also important in computational science. In~\cite{Jar} it was shown that the class of nested canalizing Boolean functions
is identical to the class of so-called unate cascade Boolean functions, which has been studied extensively in engineering
and computer science. It was shown, for instance, in \cite{But} that this class has the
property that it corresponds exactly to the class of Boolean functions with corresponding
binary decision diagrams of shortest average path length. Thus, a more
detailed mathematical study of nested canalizing functions might have applications to
problems in engineering as well.

In this paper, we carry out such a detailed study for the case of Boolean nested canalizing functions, 
obtaining a more explicit characterization than the one obtained in \cite{Abd2}. We introduce a new concept, the
\emph{layer number}, leading to a finer classification and, in particular, an explicit formula for the number of nested canalizing functions. This provides a closed form solution to the recursive formula derived in \cite{Abd2}, which may be of independent
mathematical interest. We also study standard properties of Boolean functions, such as variable activity and
average sensitivity. In particular, we obtain a formula for the average sensitivity of
a nested canalizing function with $n$ variables, and show that, for all $n$, it lies between $\frac{n}{2^{n-1}}$ as a lower bound
and $2$ as upper bound, which is much smaller
than $\frac{n}{2}$, the average sensitivity of a random Boolean function in $n$ variables. 
This can be interpreted as providing a theoretical justification for why
Boolean networks with nested canalizing rules are stable.
We also find a formula for the Hamming weight (the number of 1's in its truth table) of a nested canalizing function. 
Finally, we conjecture that a nested canalizing function with $n$ variables has maximal average sensitivity
if it has the maximal layer number $n-1$. Based on this result, we conjecture the tight upper bound $\frac{4}{3}$ for this value.
The paper is organized as follows. We first review existing results on nested
canalizing functions and networks, after which we introduce some definitions and notation.
The subsequent sections contain the main results of the paper.
\section{Background}
In \cite{Mor} it was shown that the dynamics
of a Boolean network which operates according to canalizing rules is robust
with regard to small perturbations. In \cite{Win2},  an exact formula was
derived for the number of Boolean canalizing functions. In \cite{Yua2}, 
the definition of canalizing functions was generalized to any finite field $\mathbb{F}_{q}$,
where $q$ is a power of a prime.  Both exact formulas and
asymptotic values for the number of generalized canalizing functions were obtained.

One important characteristic of (nested) canalizing functions is that they
exhibit a stabilizing effect on the dynamics of a Boolean network. That is, small
perturbations of an initial state do not grow larger over time and eventually
end up in the same attractor as the initial state. This stability is typically
measured using so-called Derrida plots which monitor the Hamming distance
between a random initial state and its perturbed state as both evolve over
time. If the Hamming distance decreases over time, the system is considered
stable. The slope of the Derrida curve is used as a numerical measure of
stability. Roughly speaking, the phase space of a stable system has few
components and the limit cycle of each component is short.

In \cite{Kau3}, the authors studied the dynamics of nested canalizing Boolean
networks over a variety of dependency graphs. That is, for a given random
graph on $n$ nodes, where the in-degree of each node is chosen at random
between $0$ and $k$, for $k\leq n$, a nested canalizing function is assigned
to each node in the in-degree variables of that node. The dynamics of
these networks was then analyzed and the stability measured using Derrida
plots. It was shown there that nested canalizing networks are remarkably stable
regardless of the in-degree distribution and that the stability increases as
the average number of inputs of each node increases.

Most published molecular networks are given in the form of a wiring diagram,
or dependency graph, constructed from experiments and prior published
knowledge. However, for most of the molecular species in the network, little
knowledge, if any, could be deduced about their regulatory rules, for
instance in the gene transcription networks in yeast \cite{Herr} and E. Coli
\cite{Bar}. Each one of these networks contains more than 1000 genes. Kauffman
et. al \cite{Kau2} investigated the effect of the topology of a sub-network of
the yeast transcriptional network where many of the transcriptional rules are
not known. They generated ensembles of different models where all models have
the same dependency graph. Their heuristic results imply that the dynamics of
those models which used only nested canalizing functions were far more stable
than the randomly generated models. Since it is already established that the
yeast transcriptional network is stable, this suggests that the unknown
interaction rules are very likely nested canalizing functions. Recently, a transcriptional network of yeast, with 3459
genes as well as the transcriptional networks of E. Coli (1481 genes) and B.
subtillis (840 genes) have been analyzed in a similar fashion, with similar findings \cite{Bal}.

The notion of sensitivity was introduced in~\cite{Coo}. The sensitivity of a Boolean function
of a variable $x$ is defined as the number of Hamming neighbors of $x$ on which the function value is different from that on $x$.
The average sensitivity of the function is then computed by taking the average value of the sensitivities 
of the function on all possible input values $x$.  Although the definition is straightforward, 
the sensitivity measure is understood only for a few classes of functions. In \cite {Shm2}, 
asymptotic formulas for a random monotone Boolean function are derived. 
Recently, S. Zhang \cite{Zha} found a formula  for the average sensitivity of any monotone Boolean function, and derived a tight bound. In \cite{Shm}, I. Shmulevich and S. A. Kauffman obtained the average sensitivity of Boolean functions with only one canalizing variable. In~\cite{Layne}, Layne et al. studied network stability of partially canalizing functions using the average sensitivity of such functions.
\section{Definitions and Notation}\label{2} 

To set the stage for this paper and for the sake of completeness 
we restate some well-known definitions; see, e.g., \cite{Kau2}.
Let $\mathbb{F}=\mathbb{F}_{2}$ be the field with $2$ elements, 
and let $f:\mathbb{F}^{n}\rightarrow \mathbb{F}$. It is well known
\cite{Lid} that $f$ can be expressed as a polynomial, called the algebraic
normal form (ANF) of $f$:
\begin{displaymath}
f(x_{1},\ldots,x_{n})=\bigoplus_{\substack{0\leq k_i\leq 1,\\i=1,\ldots,n}}a_{k_{1}\ldots k_{n}}{x_{1}}^{k_{1}}\cdots{x_{n}}^{k_{n}},%
\end{displaymath}
where each coefficient $a_{k_{1}\ldots k_{n}}\in\mathbb{F}$. The number $k_{1}+k_{2}+\cdots+k_{n}$ is the multivariate degree of
the term $a_{k_{1}k_{2}\ldots k_{n}}{x_{1}}^{k_{1}}{x_{2}}^{k_{2}}\cdots
{x_{n}}^{k_{n}}$ for each nonzero coefficient $a_{k_{1}k_{2}\ldots k_{n}}$. The
greatest degree of all the terms of $f$ is called its algebraic degree,
denoted by $deg(f)$.  The symbol $\oplus$ stands for addition  modulo $2$,
whereas the symbol $+$ will be reserved for addition of real numbers.
First we need a technical definition.

\begin{defn} \label{def2.1} 
The function $f(x_{1},x_{2},\ldots, x_{n})$ is essential in the variable $x_{i}$
if there exist $x_{1}^{*},\ldots,x_{i-1}^{*}$
$,x_{i+1}^{*},\ldots, x_{n}^{*}\in \mathbb F$ such that 
$$
f(x_{1}^{*},\ldots,x_{i-1}
^{*},0,x_{i+1}^{*},\ldots,x_{n}^{*})\neq f(x_{1}^{*},\ldots,x_{i-1}%
^{*},1,x_{i+1}^{*},\ldots,x_{n}^{*}).
$$
\end{defn}

The next two definitions state the requirements for Boolean functions to be canalizing, respectively
nested canalizing. The concept of \emph{canalization} in gene regulation goes back to work of the geneticist C. Waddington 
in the 1940s \cite{Wad}, who developed it as a possible answer to the question of why the outcome of embryonal 
development leads to predictable phenotypes in the face of widely varying environmental conditions. 
Canalized traits of an organisms, those that are stable under (some) environmental perturbations, 
are phenotypically expressed only in certain environments or genetic backgrounds. The regulation of the 
genes responsible for the development of such traits by other genes has to be able to buffer these perturbations.
In \cite{Kauffman1974}, S. Kauffman tried to capture the spirit of these features
in the context of Boolean network models of gene regulatory networks. In that setting, a Boolean function $f$ is
canalizing in a variable $x$, with canalizing input $a$ and canalized output $b$, if, whenever
$x$ takes on the value $a$, then $f$ outputs the value $b$, regardless of the states of the other variables
in $f$. Boolean network models built from functions with this property
have been shown to have dynamic features that match those of gene regulatory networks. 

\begin{defn}\label{def2.2} 
Let $a, b\in\mathbb{F}$. A function $f(x_{1},x_{2},\ldots,x_{n})$ is $<i:a:b>$
canalizing if $f(x_{1},\ldots,x_{i-1},a,x_{i+1},\ldots,x_{n})=b$, for all $x_{j}$, $j\neq
i$, where $i\in\{1,\ldots,n\}$.
\end{defn}
Thus, if the function receives its canalizing input $a$ for variable $x_i$, then the function obtained by
substituting $a$ for $x_i$ becomes a constant function equal to $b$.

The motivation for the next definition is the general stability of gene regulatory networks. In the context
of a Boolean representation of gene regulation, if a gene $x_i$ does not receive its canalizing input $a$, then, in principle,
the function obtained by substituting $a\oplus 1$ for $x_i$ can be a random Boolean function, with uncertain stability
properties. In order to remedy this deficiency, it was proposed in \cite{Kau2} that in this case there should be another
variable, $x_j$ that is canalizing for a particular input; and so on. The next definition captures this intuition.

\begin{defn}\label{def2.3} 
Let $f$ be a Boolean function in $n$ variables. Let $\sigma$ be
a permutation of the set $\{1,2,\ldots,n\}$. The function $f$ is a \emph{nested canalizing
function} in the variable order
$x_{\sigma(1)},\ldots,x_{\sigma(n)}$ with canalizing input values
$a_{1},\ldots,a_{n}$ and canalized values $b_{1},\ldots,b_{n}$, if it can be
represented in the form
\begin{displaymath}
f(x_{1},\ldots,x_{n})=
\left\{
\begin{array}
[c]{ll}%
b_{1} & x_{\sigma(1)}=a_{1},\\
b_{2} & x_{\sigma(1)}= \overline{ a_{1}}, x_{\sigma(2)}=a_{2},\\
b_{3} & x_{\sigma(1)}= \overline{ a_{1}}, x_{\sigma(2)}= \overline{ a_{2}},
x_{\sigma(3)}=a_{3},\\
\cdots  & \\
b_{n} & x_{\sigma(1)}= \overline{ a_{1}},\ldots,x_{\sigma(n-1)}= \overline{ a_{n-1}}, x_{\sigma(n)}=a_{n},\\
\overline{b_{n}} & x_{\sigma(1)}= \overline{ a_{1}},\ldots,x_{\sigma(n-1)}= \overline{ a_{n-1}}, x_{\sigma(n)}=\overline{
a_{n}}.
\end{array}
\right.
\end{displaymath}
Here, $\overline{a}=a\oplus 1$.

The function $f$ is nested canalizing if $f$ is nested
canalizing in the variable order $x_{\sigma(1)},\ldots,x_{\sigma(n)}$ for some
permutation $\sigma$.
\end{defn}

Let $\alpha=(a_{1},a_{2},\ldots,a_{n})$ and $\beta=(b_{1},b_{2},\ldots,b_{n})$.
We say that $f$ is $\{\sigma:\alpha:\beta\}$ NCF if it is NCF in the variable
order $x_{\sigma(1)},\ldots,x_{\sigma(n)}$ with canalizing input values
$\alpha=(a_{1},\ldots,a_{n})$ and canalized values $\beta=(b_{1},\ldots, b_{n})$.

Given a vector $\alpha=(a_{1},a_{2},\ldots,a_{n})$, we define 
\begin{displaymath}
\alpha
^{i_{1},\ldots,i_{k}}=(a_{1},\ldots,\overline{a_{i_{1}}},\ldots,\overline
{a_{i_{k}}},\ldots,a_{n}).
\end{displaymath}
Then, from the above definition, we immediately have the following result.

\begin{prop}
The function $f$ is $\{\sigma:\alpha:\beta\}$ NCF  if and only if  $f$ is
$\{\sigma:\alpha^{n}:\beta^{n}\}$ NCF.
\end{prop}

\begin{example}\label{exa2.1} 
The function $f_1(x_{1},x_{2},x_{3})=x_{1}(x_{2}\oplus 1)x_{3}\oplus 1$ is
$\{(1,2,3):(0,1,0):(1,1,1)\}$ NCF.
Actually, one can check that this function is nested canalizing in any variable order. Its truth table 
is given in Table~\ref{table_ex}.
\end{example}

\begin{table}
  \centering  
\begin{tabular}{| c  c  c | c |}\hline
 $x_1$ & $x_2$ & $x_3$ & $f_1$ \\ \hline
 0 & 0 & 0 & 1 \\
 0 & 0 & 1 & 1 \\
 0 & 1 & 0 & 1 \\
 0 & 1 & 1 & 1 \\
 1 & 0 & 0 & 1 \\
 1 & 0 & 1 & 0 \\
 1 & 1 & 0 & 1 \\
 1 & 1 & 1 & 1 \\\hline
\end{tabular}
\begin{tabular}{| c  c  c | c |}
\hline
$x_1$ & $x_2$ & $x_3$ & $f_2$ \\ \hline
 0 & 0 & 0 & 0 \\
 0 & 0 & 1 & 0 \\
 0 & 1 & 0 & 1 \\
 0 & 1 & 1 & 0 \\
 1 & 0 & 0 & 1 \\
 1 & 0 & 1 & 1 \\
 1 & 1 & 0 & 1 \\
 1 & 1 & 1 & 1 \\
\hline
\end{tabular}
  \caption{Truth table for Example~\ref{exa2.1} (left) and Example~\ref{exa2.2} (right).}
  \label{table_ex}
\end{table}

\begin{example}\label{exa2.2} 
Let $f_2(x_{1},x_{2},x_{3})=(x_{1}\oplus 1)(x_{2}(x_{3}\oplus 1)\oplus 1)\oplus 1$. This
function is
$\{(1,2,3):(1,0,1):(1,0,0)\}$ NCF. It is also $\{(1,3,2):(1,1,1):(1,0,1)\}$ NCF.
One can check that this function can be nested canalizing in only two variable
orders, namely $(x_{1},x_{2},x_{3})$ and $(x_{1},x_{3},x_{2})$. See its truth table in Table~\ref{table_ex}.
\end{example}

From the above definitions, we know that, if a function is NCF, all the variables appearing in it
must be essential. However, a constant function $b$ can be  $<i:a:b>$ NCF for any $i$ and $a$.
\section{A Detailed Categorization of Nested Canalizing Functions}\label{3}

As we will see, in a nested canalizing function, some variables are more
dominant than others. We will classify all the variables of an NCF into
different levels according to the extent of their dominance.  

\begin{defn}\label{def3.1} \cite{Win} 
A function $M(x_{1},\ldots,x_{n})$ is an extended monomial of
essential variables $x_{1},\ldots,x_{n}$ if 
$$M(x_{1},\ldots,x_{n})=(x_{1}\oplus a_{1})\cdots (x_{n}\oplus a_{n}),$$ 
where $a_{i}\in\mathbb{F}_{2}$ for $i\in\{1,\dots,n\}$.
\end{defn}

We will rewrite Theorem 3.1 in~\cite{Abd2} with more information in the main theorem of this section. Basically, we will obtain a unique (the old one is not) algebraic normal form (polynomial form). Because of the uniqueness, the enumeration of the number of nested
canalizing functions, computation of their Hamming weight, 
as well as activity and average sensitivity can be done. 
Besides, in the old form, the variable order of a nested canalizing function is not unique. 

\begin{lemma}\label{lm3.1}  
The function $f(x_{1},x_{2},\ldots,x_{n})$ is $<i:a:b>$ canalizing if and only if
$f(x_{1},\ldots,x_{n})=(x_{i}\oplus a)Q(x_{1},\ldots, x_{i-1},x_{i+1},\ldots,
x_{n})\oplus b$ for some polynomial $Q$.
\end{lemma}

\begin{proof}
Using the algebraic normal form of $f$, we rewrite it as 
\begin{displaymath}
f=x_{i}g_{1}(x_{1},\ldots,x_{i-1},x_{i+1},\ldots,x_{n})\oplus g_{0}(x_{1},\ldots,x_{i-1},x_{i+1},\ldots,x_{n}),
\end{displaymath} 
where $g_1$ and $g_0$ are the quotient and remainder of $f$ when divided by $x_i$.
Hence,   
\begin{align*}
\label{}
f(X)=f(x_{1},\ldots,x_{n})
=(x_{i} \oplus a)g_{1}(x_{1},\ldots,x_{i-1},x_{i+1},\ldots,x_{n})\\
\oplus ag_{1}(x_{1},\ldots,x_{i-1},x_{i+1},\ldots,x_{n})\\
\oplus g_{0}(x_{1},\ldots,x_{i-1},x_{i+1},\ldots,x_{n}).    
\end{align*}
Let $g_{1}(x_{1},\ldots,x_{i-1},x_{i+1},\ldots,x_{n})=Q(x_{1},\ldots,x_{i-1},x_{i+1}\ldots , x_{n})$, 
and
\begin{align*}
\label{}
r(x_{1},\ldots,x_{i-1},x_{i+1},\ldots,x_{n}) = & ag_{1}(x_{1},\ldots,x_{i-1},x_{i+1}
,\ldots,x_{n}) \\
   & \oplus g_{0}(x_{1},\ldots,x_{i-1},x_{i+1},\ldots,x_{n}). 
\end{align*}
Then
\begin{align*}
\label{}
 f(X)=f(x_{1},\ldots,x_{n})=   & (x_{i}\oplus a)Q(x_{1},\ldots, x_{i-1},x_{i+1}\ldots
,x_{n})  \\
    & \oplus r(x_{1},\ldots,x_{i-1},x_{i+1}%
,\ldots,x_{n}). 
\end{align*}
Since $f(X)$ is $<i:a:b>$ canalizing, we get that 
$$
f(x_{1},\ldots,x_{i-1}
,a,x_{i+1},\ldots,x_{n})=b
$$ 
for any $x_{1},\ldots, x_{i-1},x_{i+1},\ldots
,x_{n}$, i.e., $r(x_{1},\ldots,x_{i-1},x_{i+1}%
,\ldots,x_{n})=b$ for any $x_{1},\ldots,x_{i-1},x_{i+1}%
,\ldots,x_{n}$. So $r(x_{1},\ldots,x_{i-1},x_{i+1}%
,\ldots,x_{n})$ must be the constant
$b$. This shows necessity, and sufficiency is obvious.
\end{proof}

\begin{remark}\label{remark1}\label{re1} 
We have the following observations.
\begin{enumerate}
\item
When we contrast this lemma to the first part of Theorem 3.1 in
\cite{Abd2}, it is important to note that $x_{i}$ is not essential in $Q$. 
\item
In \cite{Yua2}, there is a general version of this result over any finite field. 
\item
In the above lemma, if $f$ is constant, then
$Q=0$.
\end{enumerate}
\end{remark}

From Definition \ref{def2.3}, we have the following result.

\begin{prop}\label{prop3.1} 
Let $f(x_{1},\ldots,x_{n})$ be $\{\sigma:\alpha:\beta\}$ NCF,
i.e., $f$ is NCF in the variable order
$x_{\sigma(1)},\ldots,x_{\sigma(n)}$ with canalizing input values
$\alpha=(a_{1},\ldots,a_{n})$ and canalized output values $\beta=(b_{1},\ldots
,b_{n})$. For $1\leq k\leq n-1$, let 
\begin{displaymath}
x_{\sigma(1)}=\overline{a_{1}},\ldots,x_{\sigma(k)}=\overline{a_{k}}.
\end{displaymath}
Then the function
$f(x_{1},\ldots,\overset{\sigma(1)}{\overline{a_{1}}},\ldots,\overset
{\sigma(k)}{\overline{a_{k}}},\ldots, x_{n})$ is $\{\sigma^{*}:\alpha
^{*}:\beta^{*}\}$ NCF on the remaining variables, where $\sigma^{*}
=x_{\sigma(k+1)},\ldots,x_{\sigma(n)}$, $\alpha^{*}=(a_{k+1},\ldots,a_{n})$
and $\beta^{*}=(b_{k+1},\ldots,b_{n})$.
\end{prop}

\begin{defn}\label{def3.2} 
Let $f(x_{1},\ldots,x_{n})$ be NCF. We call a variable $x_{i}$
a most dominant  variable of $f$ if there is a variable order
$\sigma=(x_{i},\ldots)$ such that $f$ is NCF  in this variable order.
\end{defn}

In Example \ref{exa2.1}, all three variables are most dominant. In Example
\ref{exa2.2}, only $x_{1}$ is a most dominant  variable. We have:

\begin{theorem}\label{th3.1} 
Given an NCF $f(x_{1},\ldots,x_{n})$, all variables are most
dominant if and only if 
$f=M(x_{1},\ldots,x_{n})\oplus b$, where $M$ is an extended monomial, i.e.,
$M=(x_{1}\oplus a_{1})(x_{2}\oplus a_{2})\cdots(x_{n}\oplus a_{n})$.
\end{theorem}

\begin{proof}
If $x_{1}$ is most dominant, from Lemma \ref{lm3.1}, we know there exist
$a_{1}$ and $b$ such that
$f(x_{1},x_{2},\ldots,x_{n})=(x_{1}\oplus a_{1})Q(x_{2},\ldots, x_{n})\oplus b$, i.e.,
$(x_{1}\oplus a_{1})|(f\oplus b)$. 
Now, if $x_{2}$ is also most dominant, then there exist $a_{2}$
and $b^{\prime}$ such that
$f(x_{1},a_{2},x_{3},\ldots,x_{n})=b^{\prime}$ for any $x_{1},x_{3}%
,\ldots,x_{n}$. Specifically, if $x_{1}=a_{1}$, we get
$f(a_{1},a_{2},x_{3},\ldots,x_{n})=b=b^{\prime}$. Hence, we also get
$(x_{2}\oplus a_{2})|(f\oplus b)=(x_{1}\oplus a_{1})Q(x_{2},\ldots, x_{n})$. Since $x_{1}\oplus a_{1}$
and $x_{2}\oplus a_{2}$ are coprime, we get $(x_{2}\oplus a_{2})|Q(x_{2},\ldots, x_{n})$,
hence, 
$f(x_{1},x_{2},\ldots,x_{n})=(x_{1}\oplus a_{1})(x_{2}\oplus a_{2})Q^{\prime}%
(x_{3},\ldots, x_{n})\oplus b$. Using the induction principle, the necessity is proved.
The sufficiency if evident.
\end{proof}

We are now ready to prove the main result of this section. Basically, we will obtain a new polynomial form of 
a nested canalizing function by induction. In this form, all the variables will be classified into different layers, 
with the variables in the outer layers more dominant than those in the inner layers. Variables in the 
same layer have the same level of dominance.
Each layer is an extended monomial of the corresponding variables.

\begin{theorem}\label{th2}
\label{th3.2} Given $n\geq2$, the function $f(x_{1},\ldots,x_{n})$ is nested
canalizing if and only if  it can be uniquely written as
\begin{equation}\label{eq3.1}
f(x_{1},\ldots,x_{n})=M_{1}(M_{2}(\cdots(M_{r-1}%
(M_{r}\oplus 1)\oplus 1)\cdots)\oplus 1)\oplus b,
\end{equation}
where each $M_{i}$ is an extended monomial. For $i\neq j$, their corresponding
sets of variables are disjoint.
More precisely, $M_{i}=\prod_{j=1}^{k_{i}}(x_{i_{j}}\oplus a_{i_{j}})$,
$i=1,\ldots,r$, $k_{i}\geq1$ for $i=1,\ldots,r-1$, $k_{r}\geq2$, $k_{1}%
+ \cdots + k_{r}=n$, $a_{i_{j}}\in\mathbb{F}_{2}$, $\{i_{j}|j=1,\ldots,k_{i},
i=1,\ldots,r\}=\{1,\ldots,n\}$.
\end{theorem}

\begin{proof}
We use induction on $n$.
When $n=2$, there are 16 Boolean functions, 8 of which are NCFs, namely
$$
(x_{1}\oplus a_{1})(x_{2}\oplus a_{2})\oplus c=M_{1}\oplus 1\oplus b, 
$$
where $b=1\oplus c$ and $M_{1}%
=(x_{1}\oplus a_{1})(x_{2}\oplus a_{2})$.

If $(x_{1}\oplus a_{1})(x_{2}\oplus a_{2})\oplus c=(x_{1}\oplus{a_{1}}^{\prime})(x_{2}\oplus{a_{2}%
}^{\prime})\oplus c^{\prime}$, then, by equating coefficients, we immediately obtain
$a_{1}={a_{1}}^{\prime}$, $a_{2}={a_{2}}^{\prime}$ and $c=c^{\prime}$. So
uniqueness holds.
We have proved that Equation \ref{eq3.1} holds for $n=2$, where $r=1$.

Assume now that Equation \ref{eq3.1} is true for any nested canalizing
function which has at most $n-1$ essential variables.
Consider a nested canalizing function $f(x_{1},\ldots,x_{n})$.
Suppose $x_{\sigma(1)},\ldots,x_{\sigma(k_{1})}$ are all most dominant
canalizing variables of $f$, for $1\leq k_{1}\leq n$.

\emph{Case 1}: $k_{1}=n$. Then, by Theorem \ref{th3.1}, the conclusion is true with $r=1$.

\emph{Case 2}: $k_{1}<n$. Then, with the same arguments as in Theorem \ref{th3.1}, we get that
$f=M_{1}g\oplus b$, where
$M_{1}=(x_{\sigma(1)}\oplus a_{\sigma(1)})\cdots(x_{\sigma(k_1)}\oplus a_{\sigma(k_1)})$. Let
$x_{\sigma(1)}=\overline{a_{\sigma(1)}},\ldots, x_{\sigma(k_1)}=\overline
{a_{\sigma(k_1)}}$ in $f$, then $f=g\oplus b$ (hence, $g$) will also be nested canalizing  
in the remaining variables, by Proposition \ref{prop3.1}. 

Since
$g$ has $n-k_{1}\leq n-1$ variables, by the induction assumption, we get that
$g=M_{2}(M_{3}(\cdots(M_{r-1}(M_{r}\oplus 1)\oplus 1)\cdots)\oplus 1)\oplus b_{1}$. It follows that
$b_{1}$ must be $1$. Otherwise, all the variables in $M_{2}$ will also be
most dominant variables of $f$. This completes the proof.
\end{proof}

Because each nested canalizing function can be uniquely written in the form  \ref{eq3.1} and the number $r$ is
uniquely determined by $f$, we can make the following definition. 

\begin{defn}\label{def3.3} 
For a nested canalizing function $f$, written in the form \ref{eq3.1}, the number $r$ will
be called its \emph{layer number}. Essential variables of $M_{1}$ will be called
most dominant variables (canalizing variables), and are part of the first layer of $f$.
Essential variables of $M_{2}$ will be
called second most dominant variables and are part of the second layer; etc.
\end{defn}

The function in Example \ref{exa2.1} has layer number 1, and the function in
Example \ref{exa2.2} has layer number 2.

\begin{remark}\label{remark2}
We make some remarks on Theorem \ref{th2}. 
\begin{enumerate}
\item
It is impossible that $k_r=1$. Otherwise, $M_r\oplus 1$ will be a factor of 
$M_{r-1}$, which means that the layer number is $r-1$. Hence, $k_r\geq 2$.
\item
If variable $x_i$ is in the first layer, and $x_i\oplus a_i$ is a factor of $M_1$, 
then this nested canalizing function is $<i:a_i:b>$ canalizing.
We simply say that $x_i$ is a canalizing variable.
\end{enumerate}
\end{remark}

From the previous examples, we know that a function can be nested canalizing for different variable orders, 
but only the variables in the same layer can be reordered. More precisely, we have the following result.

\begin{cor}\label{CoA}
If $\sigma$ and $\eta$ are two permutations on $\{1,\ldots,n\}$, 
and $f$ is both $\{\sigma: \alpha: \beta\}$ NCF and $\{\eta: \alpha': \beta'\}$ NCF, 
then we have 
$$
\{x_{\sigma(1)},\ldots, x_{\sigma(k_1)}\}=\{x_{\eta(1)},\ldots, x_{\eta(k_1)}\}, \ldots,
$$
$$
\{x_{\sigma(k_1+\cdots+k_{r-1}+1)},\ldots, x_{\sigma(n)}\}=\{x_{\eta(k_1+\cdots+k_{r-1}+1)},\ldots, x_{\eta(n)}\}.
$$ 
We also have $\beta=\beta'=(b_1,\ldots,b_n),$ and 
$$
b_1=\cdots =b_{k_1}, \ldots , b_{k_1+\cdots+k_{r-1}+1}=\cdots =b_n.
$$ 
Furthermore, 
$$
b_1\neq b_{k_1+1}\neq b_{k_1+k_2+1}\neq \cdots \neq b_{k_1+\cdots +k_{r+1}+1}.
$$
\end{cor}

\begin{proof}
These results all follow from the expression for $f$ in Theorem \ref{th3.2}.
\end{proof}

From this corollary we can determine the layer number of any nested canalizing function by its canalized value. 
For example, if $f$ is  nested canalizing with canalized value 
$(1,1,1,0,0,0,1,0,0,1,1)$ ($n=11$), then its layer number is $5$.

Let $\mathbb{NCF}(n,r)$ denote the set of all nested
canalizing functions in $n$ variables with layer number $r$, and let $\mathbb{NCF}(n)$ denote
the set of all nested canalizing functions in $n$ variables.

\begin{cor}
\label{co3.1} 
For $n\geq2$,
\[
|\mathbb{NCF}(n,r)|=2^{n+1}\sum_{\substack{k_{1}+\cdots+k_{r}=n\\k_{i}%
\geq1,i=1,\ldots,r-1, k_{r}\geq2}}\binom{n}{k_{1},\ldots,k_{r-1}},%
\]
and
\[
|\mathbb{NCF}(n)|=2^{n+1}\sum_{\substack{r=1}}^{n-1}\sum_{\substack{k_{1}%
+\cdots+k_{r}=n\\k_{i}\geq1,i=1,\ldots,r-1, k_{r}\geq2}}\binom{n}{k_{1}%
,\ldots,k_{r-1}}%
\]
where the multinomial coefficient $\binom{n}{k_{1},\ldots,k_{r-1}}$ is equal to $\frac
{n!}{k_{1}!\cdots k_{r}!}$.
\end{cor}

\begin{proof}
It follows from Equation \ref{eq3.1}, that for each choice of $k_{1},\ldots, k_{r}$, with
$k_{1}+\cdots+k_{r}=n$, $k_{i}\geq1$, $i=1,\ldots,r-1$ and
$k_{r}\geq2$,
there are $2^{k_{j}}\binom{n-k_{1}-\cdots-k_{j-1}}{k_{j}}$ ways to form
$M_{j}$, $j=1,\ldots,r$.

Note that we have two choices for $b$.
Hence,
\[
|\mathbb{NCF}(n,r)|=2\sum_{\substack{k_{1}+\cdots+k_{r}=n\\k_{i}%
\geq1,i=1,\ldots,r-1, k_{r}\geq2}}2^{k_{1}+\cdots+k_{r}}
\]

\[
\binom{n}{k_{1}}%
\binom{n-k_{1}}{k_{2}}\cdots\binom{n-k_{1}-\cdots-k_{r-1}}{k_{r}}%
\]

\[
=2^{n+1}\sum_{\substack{k_{1}+\cdots+k_{r}=n\\k_{i}\geq1,i=1,\ldots,r-1,
k_{r}\geq2}}
\]

\[
\frac{n!}{(k_{1})!(n-k_{1})!}\frac{(n-k_{1})!}{(k_{2}%
)!(n-k_{1}-k_{2})!}\cdots
\]

\[
\frac{(n-k_{1}-\cdots-k_{r-1})!}{k_{r}!(n-k_{1}%
-\cdots-k_{r})!}%
\]

\[
=2^{n+1}\sum_{\substack{k_{1}+\cdots+k_{r}=n\\k_{i}\geq1,i=1,\ldots,r-1,
k_{r}\geq2}}\frac{n!}{k_{1}!k_{2}!\cdots k_{r}!}=
\]

\[
=2^{n+1}\sum_{\substack{k_{1}%
+\cdots+k_{r}=n\\k_{i}\geq1,i=1,\ldots,r-1, k_{r}\geq2}}\binom{n}{k_{1}%
,\ldots,k_{r-1}}.
\]
Since $\mathbb{NCF}(n)=\bigcup_{r=1}^{n-1}\mathbb{NCF}(n,r)$ and
$\mathbb{NCF}(n,i)\bigcap\mathbb{NCF}(n,j)=\phi$ when $i\neq j$, we get the
formula for $|\mathbb{NCF}(n)|$.
\end{proof}

As examples, one can check that $|\mathbb{NCF}(2)|=8$, $|\mathbb{NCF}(3)|=64$,
$|\mathbb{NCF}(4)|=736$, $|\mathbb{NCF}(5)|=10624$, $\ldots$.
These results are consistent with those in \cite{Ben, Sas}.
By equating our formula to the recursive relation in \cite{Ben, Sas}, we have

\begin{cor}
\label{co3.2} The solution of the nonlinear recursive sequence%
\[
a_{2}=8, a_{n}=\sum_{r=2}^{n-1}\binom{n}{r-1}2^{r}a_{n-r+1}+2^{n+1} , n\geq3
\]
is
\[
a_{n}=2^{n+1}\sum_{\substack{r=1}}^{n-1}\sum_{\substack{k_{1}+\cdots
+k_{r}=n\\k_{i}\geq1,i=1,\ldots,r-1, k_{r}\geq2}}\binom{n}{k_{1}%
,\ldots,k_{r-1}}.%
\]

\end{cor}


\section{Hamming Weight, Activity, and Average Sensitivity}\label{sensitivies}

A Boolean function is called \emph{balanced} if it takes the value 1 on exactly half the
states (and 0 on the other half). In other words, its Hamming weight (the number of 1's in its truth table) is $2^{n-1}$, where $n$ is the number of variables. Hence, there are $\binom{2^n}{2^{n-1}}$ balanced  Boolean functions. 
It is easy to show that a Boolean function with canalizing variables is not balanced, 
i.e., is biased, actually, very biased. For example, the two constant functions are trivially canalizing,
and they are also the most biased functions. Extended monomial functions are 
the second most biased since for any of them, only one value is nonzero. 
But biased functions may have no canalizing variables. 
For example, $f(x_1,x_2,x_3)=x_1x_2x_3\oplus x_1x_2\oplus x_1x_3\oplus x_2x_3$ 
is biased but without canalizing variables.

In Boolean functions, some variables have greater influence over the output of the function than other variables. 
To formalize this, a  concept called \emph{activity} was introduced. Let $x=(x_1,\ldots,x_n)$, and
$$
\frac{\partial f(x)}{\partial x_i}=f(x_1,\ldots,x_i\oplus 1,\ldots,x_n)\oplus f(x_1,\ldots,x_i,\ldots,x_n).
$$ 
The \emph{activity} of variable $x_i$ is defined as
\begin{equation}\label{act1}
\lambda_i^f=\frac{1}{2^n}\sum_{(x_1,\ldots,x_n)\in \mathbb{F}_2^n}\frac{\partial f(x_1,\ldots,x_n)}{\partial x_i}
\end{equation}

Note that the above definition can also be written as follows:
\begin{equation}\label{act2}
\lambda_i^f=\frac{1}{2^{n-1}}\sum_{(x_1,\ldots,x_{i-1},x_{i+1},\ldots,x_n)\in \mathbb{F}_2^{n-1}}
(f(x_1,\ldots,\overset{i}{0},\ldots,x_n)\oplus f(x_1,\ldots,\overset{i}{1},\ldots,x_n))
\end{equation}
The activity of any variable in a constant function is 0. For an affine function 
$f(x_1,\ldots,x_n)=x_1\oplus \cdots\oplus x_n\oplus b$, $\lambda_i^f=1$ for any $i$. 
It is clear, for any $f$ and $i$, that we have $0\leq \lambda_i^f\leq 1$.

Another important quantity is the sensitivity of a Boolean function, which measures how sensitive the output of the function is if the input changes (This was introduced in \cite {Coo}). The sensitivity $s^f(x_1,\ldots,x_n)$ of $f$ on the input $(x_1,\ldots,x_n)$ is defined as
 the number of Hamming neighbors of $(x_1,\ldots,x_n)$ (that is, all states that have Hamming distance 1)
 on which the function value is different from $f(x_1,\ldots,x_n)$. That is,
\begin{equation*}
s^f(x_1,\ldots,x_n)=|\{i|f(x_1,\ldots,\overset{i}{0},\ldots,x_n)\neq 
\end{equation*}
\begin{equation*}
f(x_1,\ldots,\overset{i}{1},\ldots,x_n), i=1,\ldots,n \}|.
\end{equation*}
Obviously, $s^f(x_1,\ldots,x_n)=\sum_{i=1}^n\frac{\partial f(x_1,\ldots,x_n)}{\partial x_i}$.

The average sensitivity of a function $f$ is defined as
\begin{equation*}
s^f=E[s^f(x_1,\ldots,x_n)]=
\end{equation*}
\begin{equation*}
\frac{1}{2^n}\sum_{(x_1,\ldots,x_n)\in \mathbb{F}_2^n}s^f(x_1,\ldots,x_n)=\sum_{i=1}^n\lambda_i^f.
\end{equation*}
It is clear that $0\leq s^f\leq n$.
The concept of average sensitivity of a Boolean function is one of the most studied concepts 
in the analysis of Boolean functions, and has received a lot attention recently
\cite{Ama, Ber, Ber2, Bop, Che, Chr, Kel, Liu, Li, Qia, Shm2, Shm, Shp,  Sch, Vir}. 
Bernasconi has shown \cite{Ber} that a random Boolean function has average sensitivity $\frac{n}{2}$. 
This means the average value of the average sensitivities of all Boolean functions in $n$ variables
is $\frac{n}{2}$. In \cite {Shm}, Shmulevich and Kauffman calculated the activity of all the 
variables of a Boolean function with exactly one canalizing variable and 
unbiased input for the other variables. Adding all the activities, 
the average sensitivity of a Boolean function was also obtained.

First, the following observation will be useful.
We have the equality 
$$
(x_1\oplus a_1)\cdots (x_k\oplus a_k)=
\left\{
\begin{array}{ll}
1, & if \quad (x_1,\ldots,x_k)=(\overline{a_1},\ldots,\overline{a_k})\\
0, & otherwise.
\end{array}\right .
$$
That is, only one value is equal to $1$ and all the other $2^k-1$ values are $0$.

\begin{theorem}\label{th4.1}
For $n\geq 2$, let $f_1=M_1$,
\begin{displaymath}
f_r=M_{1}(M_{2}(\cdots(M_{r-1}(M_{r}\oplus 1)\oplus 1)\cdots)\oplus 1),\ r\geq 2
\end{displaymath}
where $M_i$ is same 
as in Theorem \ref{th3.2}.  Then the Hamming weight of $f_r$ is
\begin{equation}\label{eq4.3}
 W(f_r)=\sum_{j=1}^r(-1)^{j-1}2^{n-\sum_{i=1}^jk_i}
\end{equation}
The Hamming weight of $f_r\oplus 1$ is
\begin{equation}\label{eq4.4}
 W(f_r\oplus 1)=\sum_{j=0}^r(-1)^{j}2^{n-\sum_{i=1}^jk_i},
\end{equation}
where $\sum_{i=1}^0k_i$ is to be interpreted as $0$.
\end{theorem}

\begin{proof}
First, consider the Hamming weight of $f_r$.
When $r=1$, we know the result is true by the above observation.
When $r>1$, we consider two cases:

\emph{Case A}: $r=2t+1$ is odd.
 Then all the states on which $f$ evaluates to 1 will be divided into the following disjoint groups:
\begin{itemize}
\item
Group $j$ for $j=1,\ldots ,t: M_1=1, M_2=1, \ldots ,M_{2j-1}=1, M_{2j}=0$; 
\item
Group $t+1$ : $M_1=1, M_2=1, \ldots , M_{2t}=1, M_{2t+1}=M_r=1$.
\end{itemize}

In Group $j$, the number of states is 
$$
(2^{k_{2j}}-1)2^{n-k_1-\cdots-k_{2j}}=2^{n-k_1-\cdots-k_{2j-1}}-2^{n-k_1-\cdots -k_{2j}}.
$$
In Group $t+1$, the number of states is $2^{n-k_1-\cdots -k_r}=1$.

Adding all of them together, we get Equation \ref{eq4.3}.

\medskip
\emph{Case B}: $r=2t$ is even.
The proof in this case is similar, and we omit it.
 
Because  $|\{(x_1,\ldots,x_n)|f(x_1,\ldots,x_n)=0\}|+|\{(x_1,\ldots,x_n)|f(x_1,\ldots,x_n)=1\}|=2^n$,
 we know that the Hamming weight of $f_r\oplus 1$ is equal to
 \begin{equation*}
 W(f_r\oplus 1)=2^n-W(f_r)=2^n-\sum_{j=1}^r(-1)^{j-1}2^{n-\sum_{i=1}^jk_i}
 \end{equation*}
 \begin{equation*}
=\sum_{j=0}^r(-1)^{j}2^{n-\sum_{i=1}^jk_i},
 \end{equation*}
 where $\sum_{i=1}^0k_i$ should be interpreted as $0$.
\end{proof}

In the following, we will calculate the activities of the variables of any nested canalizing function. 
For this, we will use the formula for the Hamming weight of a nested canalizing function since the function in the summation 
will be reduced to a nested canalizing function, for which the first layer is a product of 
a few layers of the original nested canalizing function. 
Let $f$ be nested canalizing, written in the form of Theorem \ref{th3.2}. 
Without loss of generality, we can assume that
 $M_1=(x_1\oplus a_1)(x_2\oplus a_2)\cdots (x_{k_1}\oplus a_{k_1})$. Let
  $m_1=(x_1\oplus a_1)\cdots(x_{i-1}\oplus a_{i-1})(x_{i+1}\oplus a_{i+1})\cdots (x_{k_1}\oplus a_{k_1})$, 
  so that $M_1=(x_i\oplus a_i)m_1$.

 If $r=1$, i.e., $k_1=n$, then 
 \begin{equation*}
\lambda_i^f=\frac{1}{2^{n-1}}\sum_{(x_1,\ldots,x_{i-1},x_{i+1},\ldots,x_n)\in \mathbb{F}_2^{n-1}}
\end{equation*}
\begin{equation*}
 (f(x_1,\ldots,\overset{i}{0},\ldots,x_n)\oplus f(x_1,\ldots,\overset{i}{1},\ldots,x_n))
\end{equation*}
\begin{equation*}
=\frac{1}{2^{n-1}}\sum_{(x_1,\ldots,x_{i-1},x_{i+1},\ldots,x_n)\in \mathbb{F}_2^{n-1}}m_1=
\end{equation*}
\begin{equation*}
\frac{1}{2^{n-1}}W(m_1)=\frac{1}{2^{n-1}}.
\end{equation*}

If $1<r\leq n-1$,
then consider the activity of $x_i$ in the first layer, i.e.,  $1\leq i\leq k_1$. We have
 \begin{equation*}
\lambda_i^f=\frac{1}{2^{n-1}}\sum_{(x_1,\ldots,x_{i-1},x_{i+1},\ldots,x_n)\in \mathbb{F}_2^{n-1}}
\end{equation*}
\begin{equation*}
(f(x_1,\ldots,\overset{i}{0},\ldots,x_n)\oplus f(x_1,\ldots,\overset{i}{1},\ldots,x_n))
\end{equation*}
\begin{equation*}
=\frac{1}{2^{n-1}}\sum_{(x_1,\ldots,x_{i-1},x_{i+1},\ldots,x_n)\in \mathbb{F}_2^{n-1}}
\end{equation*}
\begin{equation*}
m_1(M_{2}(\cdots(M_{r-1}%
(M_{r}\oplus 1)\oplus 1)\cdots)\oplus 1)
\end{equation*}
\begin{equation*}
=\frac{1}{2^{n-1}}W(m_1(M_{2}(\cdots(M_{r-1}%
(M_{r}\oplus 1)\oplus 1)\cdots)\oplus 1)).
\end{equation*}
$=\left\{
\begin{array}{ll}
\frac{1}{2^{n-1}}\sum_{j=1}^r(-1)^{j-1}2^{n-1-(\sum_{i=1}^jk_i-1)}, & if\quad k_1>1\\
\frac{1}{2^{n-1}}\sum_{j=0}^{r-1}(-1)^{j}2^{n-1-\sum_{i=1}^jk_{i+1}}, & if\quad k_1=1.
\end{array}\right.$

$=\left\{
\begin{array}{ll}
\frac{1}{2^{n-1}}\sum_{j=1}^r(-1)^{j-1}2^{n-\sum_{i=1}^jk_i}, & if\quad k_1>1\\
\frac{1}{2^{n-1}}\sum_{j=0}^{r-1}(-1)^{j}2^{n-\sum_{i=0}^jk_{i+1}}, & if\quad k_1=1.
\end{array}\right.$

$=\left\{
\begin{array}{ll}
\frac{1}{2^{n-1}}\sum_{j=1}^r(-1)^{j-1}2^{n-\sum_{i=1}^jk_i}, & if\quad k_1>1\\
\frac{1}{2^{n-1}}\sum_{j=0}^{r-1}(-1)^{j}2^{n-\sum_{i=1}^{j+1}k_{i}}, & if\quad k_1=1.
\end{array}\right.$

$=\left\{
\begin{array}{ll}
\frac{1}{2^{n-1}}\sum_{j=1}^r(-1)^{j-1}2^{n-\sum_{i=1}^jk_i}, & if\quad k_1>1\\
\frac{1}{2^{n-1}}\sum_{j=1}^r(-1)^{j-1}2^{n-\sum_{i=1}^jk_i}, & if\quad k_1=1.
\end{array}\right.$

$=\frac{1}{2^{n-1}}\sum_{j=1}^r(-1)^{j-1}2^{n-\sum_{i=1}^jk_i}$
by Theorem \ref{th4.1}. Note, in the above, $k_1=1$ means $m_1=1$,  so
we used Equation \ref{eq4.4} with layer number $r-1$, and the first layer is $M_2$ for $n-1$ variable functions.

Now let us consider the variables in the $l-th$  layer, i.e., $x_i$ is an essential variable of $M_l$, $2\leq l\leq r-1$. We have
$M_l=(x_i+a_i)m_l$ and 
\begin{equation*}
\lambda_i^f=\frac{1}{2^{n-1}}\sum_{(x_1,\ldots,x_{i-1},x_{i+1},\ldots,x_n)\in \mathbb{F}_2^{n-1}}
\end{equation*}
\begin{equation*}
(f(x_1,\ldots,\overset{i}{0},\ldots,x_n)\oplus f(x_1,\ldots,\overset{i}{1},\ldots,x_n))
\end{equation*}
\begin{equation*}
=\frac{1}{2^{n-1}}\sum_{(x_1,\ldots,x_{i-1},x_{i+1},\ldots,x_n)\in \mathbb{F}_2^{n-1}}
\end{equation*}
\begin{equation*}
M_1\cdots M_{l-1}m_l(M_{l+1}(\cdots (M_r\oplus 1)\cdots )\oplus 1).
\end{equation*}
\begin{equation*}
=\frac{1}{2^{n-1}}\sum_{j=1}^{r-l+1}(-1)^{j-1}2^{n-1-((k_1+\cdots + k_{l}-1)+k_{l+1}+\cdots  +k_{j+l-1}))}=
\end{equation*}
\begin{equation*}
\frac{1}{2^{n-1}}\sum_{j=1}^{r-l+1}(-1)^{j-1}2^{n-\sum_{i=1}^{j+l-1}k_i}
\end{equation*}
by Equation \ref{eq4.3} in Theorem \ref{th4.1}. Note that $M_1\cdots M_{l-1}m_l$ is the first layer, $M_{l+1}$ is the second layer, etc.

Let $x_i$ be the variable in the last layer $M_{r}$, then we have
\begin{equation*}
=\frac{1}{2^{n-1}}\sum_{(x_1,\ldots,x_{i-1},x_{i+1},\ldots,x_n)\in \mathbb{F}_2^{n-1}}M_1M_2\cdots\ M_{r-1}m_r=\frac{1}{2^{n-1}}.
\end{equation*}
Variables in the same layer have the same activities, so we use $A_l^f$ to stand for the activity number of each variable in the $lth$ layer $M_l$, $1\leq l\leq r$. We find that  
the formula of $A_l^f$ for $2\leq l\leq r-1$ is also true when $l=r$ or $r=1$. The next theorem summarizes these.

\begin{theorem}\label{th4.2}
Let $f$ be a nested canalizing function, written as in Theorem \ref{th3.2}.
Then the activity of each variable in the $lth$  layer , $1\leq l\leq r$, is
\begin{equation}\label{4.5}
A_l^f=\frac{1}{2^{n-1}}\sum_{j=1}^{r-l+1}(-1)^{j-1}2^{n-\sum_{i=1}^{j+l-1}k_i} .
\end{equation}
The average sensitivity of $f$ is 
\begin{equation}\label{eq4.6}
s^f=\sum_{l=1}^rk_lA_l^f=\frac{1}{2^{n-1}}\sum_{l=1}^r k_l\sum_{j=1}^{r-l+1}(-1)^{j-1}2^{n-\sum_{i=1}^{j+l-1}k_i} .
\end{equation}
\end{theorem}
 
 Next, we analyze the formulas in Theorem \ref{th4.2}.
 
  \begin{cor}\label{co4.1}
 If $n\geq 3$, then  $A_1^f>A_2^f>\cdots >A_r^f$, and $\frac{n}{2^{n-1}}\leq  s^f < 2- \frac{1}{2^{n-2}}$ .
 \end{cor}
 
 \begin{proof}
\begin{equation*}
A_l^f=\frac{1}{2^{n-1}}\sum_{j=1}^{r-l+1}(-1)^{j-1}2^{n-\sum_{i=1}^{j+l-1}k_i}
\end{equation*}
\begin{equation*}
=\frac{1}{2^{n-1}}(2^{n-k_1-\cdots -k_l}-2^{n-k_1-\cdots -k_{l+1}}+\cdots (-1)^{r-l})
\end{equation*}
Since the sum is an alternating decreasing  sequence and $k_{l+1}\geq 1$, we have
\begin{equation*}
\frac{1}{2^{n-1}}(2^{n-k_1-\cdots -k_l-1})\leq \frac{1}{2^{n-1}}(2^{n-k_1-\cdots -k_l}-2^{n-k_1-\cdots -k_{l+1}})
\end{equation*}
\begin{equation*}
< A_l^f< \frac{1}{2^{n-1}}(2^{n-k_1-\cdots -k_l}).
\end{equation*}
Hence,
\begin{equation*}
 A_{l+1}^f< \frac{1}{2^{n-1}}(2^{n-k_1-\cdots -k_{l+1}})
\end{equation*}
\begin{equation*}
 \leq \frac{1}{2^{n-1}}(2^{n-k_1-\cdots -k_l-1})<A_l^f.
\end{equation*}

We have 
\begin{equation*}
k_1A_1^f=\frac{k_1}{2^{n-1}}(2^{n-{k_1}}-2^{n-{k_1}-k_2}+\cdots (-1)^{r-1});
\end{equation*}
\begin{equation*}
k_2A_2^f=\frac{k_2}{2^{n-1}}(2^{n-k_1-k_2}-2^{n-k_1-k_2-k_3}+\cdots (-1)^{r-2});
\end{equation*}
\begin{equation*}
\cdots \cdots
\end{equation*}
\begin{equation*}
k_lA_l^f=\frac{k_l}{2^{n-1}}(2^{n-k_1-\cdots-k_l}-2^{n-k_1-\cdots-k_l-k_{l+1}}-\cdots (-1)^{r-l});
\end{equation*}
\begin{equation*}
\cdots 
\end{equation*}
\begin{equation*}
k_rA_r^f=\frac{k_r}{2^{n-1}}.
\end{equation*}
Hence, $s^f=\sum_{l=1}^rk_lA_l^f\geq \frac{k_1}{2^{n-1}}+\frac{k_2}{2^{n-1}}+\cdots +\frac{k_r}{2^{n-1}}=\frac{n}{2^{n-1}}$, so we know 
that the nested canalizing functions with layer number 1 have minimal average sensitivity.
On the other hand, 
$s^f=\sum_{l=1}^rk_lA_l^f<\frac{k_1}{2^{n-1}}2^{n-k_1}+\frac{k_2}{2^{n-1}}2^{n-k_1-k_2}+\cdots +\frac{k_l}{2^{n-1}}2^{n-k_1-\cdots -k_l}+\dots+\frac{k_r}{2^{n-1}}=U(k_1,\cdots,k_r)$,
where $k_1+\cdots +k_r=n$, $k_i\geq 1$, $i=1,\ldots, r-1$ and $k_r\geq 2$. We will find the maximal value of $U(k_1,\ldots,k_r)$ in the following.

First, we claim $k_r=2$ if $U(k_1,\ldots,k_r)$ reaches its maximal value.  
Because, if $k_r$ is increased by $1$, and the last term makes $\frac{1}{2^{n-1}}$ more contributions to $U(k_1,\ldots,k_r)$, 
then there exists $l$ such that $k_l$ will be decreased by $1$ ($k_1+\cdots+k_r=n$), hence 
\begin{equation*}
\frac{k_l}{2^{n-1}}2^{n-k_1-\cdots -k_l}
\end{equation*}
will be decreased more than $\frac{1}{2^{n-1}}$.
Now, it is obvious that $\frac{k_1}{2^{n-1}}2^{n-k_1}$ attains its maximal value only when 
$k_1=1$ or $2$, but $k_1=1$ will be the choice since it also makes all the other terms greater.
Likewise, $\frac{k_2}{2^{n-1}}2^{n-k_1-k_2}$ attains its maximal value  when $k_1=k_2=1$ or $k_1=1$ and $k_2=2$; 
again, $k_2=1$ is the best choice  to make all the other terms greater.
 In general, if $k_1=\cdots=k_{l-1}=1$, then $\frac{k_l}{2^{n-1}}2^{n-k_1-\cdots -k_l}$  attains its maximal value 
 when $k_l=1$, where $1\leq l\leq r-1$.
 In summary, we have shown that $U(k_1,\ldots,k_r)$ reaches its maximal value 
 when $r=n-1$, $k_1=\cdots=k_{n-2}=1$, $k_{n-1}=2$, and 
\begin{align*}
\label{}
Max [ U(k_1,\ldots,k_r)=U(1,\ldots,1,2 )]=   \\
    \frac{1}{2^{n-1}}(2^{n-1}+2^{n-2}+\cdots+ 2^{2}+2)=2-\frac{1}{2^{n-2}}. 
\end{align*}
\end {proof}

\begin{remark}
The minimal value of average sensitivity approaches $0$ and the maximal value of $U(k_1,\ldots,k_r)$ approaches $2$ as $n\rightarrow \infty$. Hence, $0<s^f<2$ for any NCF with an arbitrary number of variables. 
\end{remark}

In the following, we evaluate Equation \ref{eq4.6} for some parameters $k_1,\ldots,k_r$.

\begin{lemma}\label{lm4.2}
\begin{enumerate}
\item
If $r=n-1, k_1=\cdots=k_{n-2}=1, k_{n-1}=2$, then $s^f=\frac{4}{3}-\frac{3+(-1)^n}{3\times 2^n}$.
\item
Given $n\geq 4, r=n-2, k_1=\ldots=k_{n-3}=1, k_{n-2}=3$, then $s^f=\frac{4}{3}-\frac{9+5(-1)^{n-1}}{3\times 2^n}$.
\item
If $n$ is even and $n\geq 6 $, $r=\frac{n}{2}$, $k_1=1$, $k_2=\cdots=k_{\frac{n}{2}-1}=2$, $k_{\frac{n}{2}}=3$, 
then $s^f=\frac{4}{3}-\frac{4}{3\times 2^n}$. Hence, these three cardinalities are equal if $n$ is even.
\end{enumerate}
\end{lemma}

\begin{proof}
When $r=n-1$, then $k_1=\cdots=k_{n-2}=1$, $k_{n-1}=2$ by Equation \ref{eq4.6}. We have
\begin{equation*}
s^f=\sum_{l=1}^{r}k_lA_l^f=\frac{1}{2^{n-1}}\sum_{l=1}^{n-1} k_l\sum_{j=1}^{n-l}(-1)^{j-1}2^{n-\sum_{i=1}^{j+l-1}k_i}
\end{equation*}
\begin{equation*}
=\frac{1}{2^{n-1}}\sum_{l=1}^{n-1} k_l(\sum_{j=1}^{n-l-1}(-1)^{j-1}2^{n-j-l+1}+(-1)^{n-l-1})
\end{equation*}
\begin{equation*}
=\frac{1}{2^{n-1}}\sum_{l=1}^{n-1} k_l(\frac{1}{3}2^{n-l+1}+\frac{1}{3}(-1)^{n-l})
\end{equation*}
\begin{equation*}
=\frac{1}{2^{n-1}}(\sum_{l=1}^{n-2} (\frac{1}{3}2^{n-l+1}+\frac{1}{3}(-1)^{n-l})+2)=\frac{4}{3}-\frac{3+(-1)^n}{3\times 2^n}.
\end{equation*}
 The other two formulas are also routine applications of Equation \ref{eq4.6}.
\end{proof}

Based on our numerical calculations, Lemma \ref{lm4.2}, and the proof of Corollary \ref{co4.1}, we can make the
following conjecture.

\begin{conj}\label{conj4.1}
The maximal value of $s^f$ is $s^f=\frac{4}{3}-\frac{3+(-1)^n}{3\times 2^n}$. It will be reached if the nested
canalizing function has the maximal layer number $n-1$, 
i.e., if $r=n-1$, $k_1=\cdots=k_{n-2}=1$, $k_{n-1}=2$. When $n$ is even, this maximal value is also reached by a nested
canalizing function with parameters 
$n\geq 4$, $r=n-2$, $k_1=\cdots=k_{n-3}=1$, $k_{n-2}=3$ or $n\geq 6 $, $r=\frac{n}{2}$, $k_1=1$, $k_2=\cdots=k_{\frac{n}{2}-1}=2$ and $k_{\frac{n}{2}}=3$.
\end{conj}

\begin{remark}\label{re4.2}
When $n=6$, the nested canalizing function with $k_1=1$, $k_2=2$, $k_3=1$, and 
$k_4=2$ also has the maximal average sensitivity $\frac{21}{16}$. But this can not be generalized. If the above conjecture is true, then we have $0<s^f<\frac{4}{3}$ for any nested canalizing function with an 
arbitrary number of variables. In other words, both $0$ and $\frac{4}{3}$ are uniform tight bounds for any nested canalizing function.
\end{remark}

\section{Simulations of Network Dynamics}

The sensitivity of functions~\cite{Shm} has been shown to be a good indicator of the stability of dynamical networks constructed using random Boolean functions. We have generated random networks controlled by nested canalizing functions with fixed layer number $m$, 
where $m = 1,\ldots,n-1$. For our simulations we followed a similar approach as in \cite{Shm}. 
Starting with a Boolean network $F$, we sample pairs of random states $x$ and $y$ for $F$. 
Let $H(t)$ be the Hamming distance of $x$ and $y$, i.e. $H(t)$ is the number of bits in which $x$ and $y$ differ,
 and let $H(t+1)$ be the Hamming distance of $F(x)$ and $F(y)$, i.e. the Hamming distance of the successor states of $x$ and $y$. 
 A Derrida curve~\cite{Derrida} is a plot of $H(t+1)$ against $H(t)$ for all possible Hamming distances. Figure~\ref{layers_plot} shows Derrida plots for different layer numbers. Each Derrida curve was generated from 4096 random networks by taking the average Hamming distances.
 
As can be observed from Figure~\ref{layers_plot}, networks made up of nested canalizing functions of 
fixed layer number equal to $1$ are significantly more stable than networks constructed from nested canalizing 
functions with higher layer numbers. 
It should be noted that the class of nested canalizing functions with layer number equal to $1$ is the same family as the 
AND-NOT networks studied in~\cite{Veliz-Cuba}, equal to 
the class of extended monomials. Similarly, networks made up of nested canalizing functions with 
fixed layer number equal to $2$ are more stable than networks constructed from 
nested canalizing functions with higher layer numbers, and so on. 
Figure~\ref{layers_plot} shows that, as the layer number increases, networks becomes less stable. 
This matches our results in Section~\ref{sensitivies}.

\begin{figure}
\begin{center}
\includegraphics[width=\textwidth]{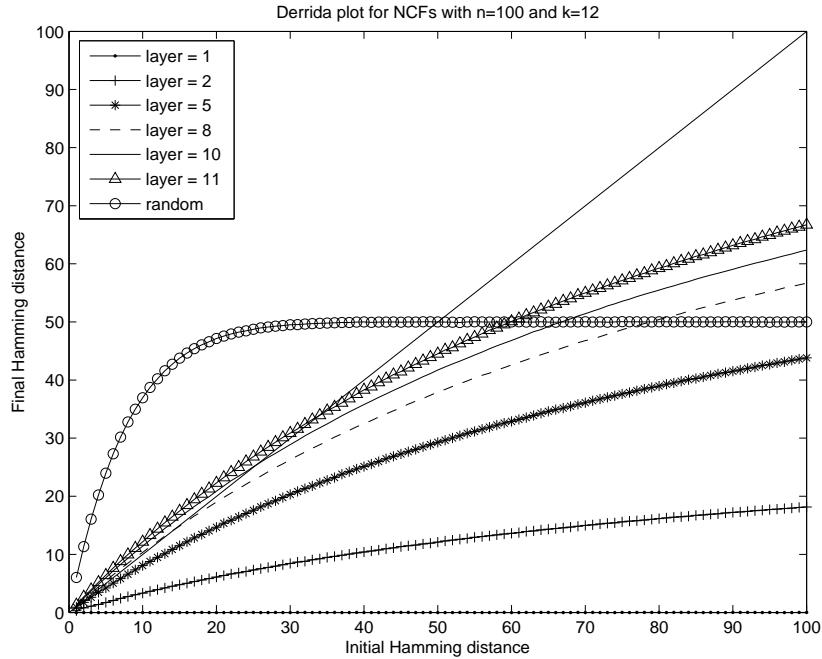}
\caption{Derrida plots for different layer numbers. Here, $n$ is the number of nodes and $k$ is the connectivity. For each layer number, a Derrida plot was generated from 4096 random networks. The $x$-axis represents
Hamming distance of pairs of states, and the $y$-axis represents Hamming distance of their images.}
\label{layers_plot}
\end{center}
\end{figure} 
\section{Conclusion}
Nested canalizing Boolean functions were inspired by structural and dynamic features of biological networks. 
In this study we took a careful look at the computational properties of Boolean functions and
of Boolean networks constructed from them. The main tool for our analysis is a particular polynomial normal form
of the Boolean functions in question. In particular, we introduced
a new invariant for nested canalizing functions, their layer number.  Using it, we
obtain an explicit formula of the number of nested canalizing functions, which improves on
the known recursive formula. Based on the
polynomial form, we also obtain a formula for the Hamming weight of a nested canalizing function. 
The activity number of each variable of a nested canalizing function is also provided with an explicit formula. 
An important result we obtain is that the average sensitivity
of any nested canalizing function is less than $2$, which provides a theoretical argument why
nested canalizing functions exhibit robust dynamic properties. 
This leads us to conjecture that the tight upper bound for the average sensitivity of any nested canalizing function is $\frac{4}{3}$.

It should be noted that all the variables in the first layer of an nested canalizing function are canalizing variables (the most dominant), 
all the variables in the second layer are the second most dominant, etc.  The fact that networks with 
low layer number are more stable than those with high layer number could be a consequence of this observation, 
since lower layer number means more dominant variables. The most extreme examples are: 
when the layer number is equal to $1$, we have proved that such a function has minimal average sensitivity, 
therefore we also conjecture that the function with layer number number equal to $n-1$ (the maximal layer number) has 
maximal average sensitivity. Hence, it should be reasonable that the corresponding networks show similar behavior.
\section*{Acknowledgments}
The authors thank the referees for insightful comments that have improved the manuscript.


\end{document}